\documentclass[sn-mathphys-num]{sn-jnl}% Math and Physical Sciences Numbered Reference Style 
\usepackage{graphicx}%
\usepackage{multirow}%
\usepackage{amsmath,amssymb,amsfonts}%
\usepackage{amsthm}%
\usepackage{mathrsfs}%
\usepackage[title]{appendix}%
\usepackage{xcolor}%
\usepackage{textcomp}%
\usepackage{manyfoot}%
\usepackage{booktabs}%
\usepackage{algorithm}%
\usepackage{algorithmicx}%
\usepackage{algpseudocode}%
\usepackage{listings}%

\newcommand{\foa}{\,\forall\,}
\newcommand{\eps}{\varepsilon}
\newcommand{\emp}{\varnothing}

\newcommand{\R}{\mathbb{R}}
\newcommand{\Z}{\mathbb{Z}}

\newcommand{\N}{\mathbb{N}}

\newcommand{\C}{\mathbb{C}}
\renewcommand{\H}{\mathbb{H}}

\newcommand{\tri}[1]{\left<{}#1\right>}

\newcommand{\pfrac}[2][]{\frac{\partial {}#1}{\partial {}#2}}

\newcommand{\id}{\mathrm{id}}
\newcommand{\vol}{\mathrm{vol}}

\newcommand{\Ric}{\operatorname{Ric}}

\newcommand{\dom}{\operatorname{dom}}
\newcommand{\Spec}{\operatorname{Spec}}

\newcommand{\HH}{\mathcal{H}}

\newcommand{\EEE}{\mathscr{E}}

\theoremstyle{thmstyleone}
\newtheorem{thm}{Theorem}[section]
\newtheorem{prop}[thm]{Proposition}
\newtheorem{lemma}[thm]{Lemma}
\newtheorem{cor}[thm]{Corollary}

\theoremstyle{thmstylethree}
\newtheorem{defi}[thm]{Definition}

\theoremstyle{thmstyletwo}
\newtheorem{example}[thm]{Example}
\newtheorem{rem}[thm]{Remark}

\begin{document}

\author[1]{\fnm{Bobo} \sur{Hua}}
\email{bobohua@fudan.edu.cn}
\equalcont{These authors contributed equally to this work.}

\author[2]{\fnm{Florentin} \sur{M\"unch}}
\email{cfmuench@gmail.com}
\equalcont{These authors contributed equally to this work.}

\author*[3]{\fnm{Haohang} \sur{Zhang}}
\email{zhanghh22@mails.tsinghua.edu.cn}
\equalcont{These authors contributed equally to this work.}

\affil[1]{\orgdiv{School of Mathematical Sciences, LMNS and Shanghai Center for
		Mathematical Sciences}, \orgname{Fudan University}, \orgaddress{\city{Shanghai}, \postcode{200438}, \country{China}}}

\affil[2]{\orgdiv{Faculty of Mathematics and Computer Science}, \orgname{Leipzig University}, \orgaddress{\city{Leipzig}, \postcode{04109}, \country{Germany}}}

\affil*[3]{\orgdiv{Yau Mathematical Science Center}, \orgname{Tsinghua University}, \orgaddress{\city{Beijing}, \postcode{100084}, \country{China}}}

\title[Inequalities between Dirichlet and Neumann Eigenvalues on Surfaces]{Inequalities between Dirichlet and Neumann Eigenvalues on Surfaces}
\date{\today}

%\tableofcontents	

\abstract{
For a bounded Lipschitz domain $\Sigma$ in a Riemannian surface $M$ satisfying certain curvature condition, we prove that $$\mu_{3-\beta_1} \leq \lambda_{1},$$ where $\mu_k$ ($\lambda_k$ resp.) is the $k$-th Neumann (Dirichlet resp.) Laplacian eigenvalue on $\Sigma$ and $\beta_1$ is the first Betti number of $\Sigma.$ If $\Sigma$ is smooth and simply connected, we can further derive the strict inequality
$
\mu_{3}< \lambda_{1}.
$
This extends previous results on the Euclidean space to various curved surfaces, including the flat cylinder, the hyperbolic plane, hyperbolic cusp, collar, funnel, and minimal surfaces such as catenoid and helicoid. The novelty of the paper lies in comparing Dirichlet and Neumann Laplacian eigenvalues via the variational principle of the Hodge Laplacian on $1$-forms on a surface, extending the variational principle on vector fields in the Euclidean plane as developed by Rohleder \cite{R2023}. The comparison is reduced to the existence of a distance function with appropriate curvature conditions on its level sets.
}



\maketitle

\section{Introduction}
	Spectral theory of Laplace operators on Euclidean spaces or Riemannian manifolds is fundamental in partial different equations, geometry and physics, which was extensively studied in the literature, see e.g. \cite{CouHil53,ReedSimon1,ReedSimon4,Chavel84,SY94}. For eigenvalues of the Laplacian on a bounded domain, the boundary condition, Dirichlet or Neumann,  is motivated by the physical setting and is of significant importance in the studies. There are many results on such eigenvalue estimates, see e.g. \cite{Weyl,S1954,Weinberger1956,PPW,Pol61,LY83,Yang91,AshBen92,Ash99,ChengYang05,ChenCheng08,BucurHenrot19,XiaWang23}.

 	For a given domain, finding intrinsic relationship between Dirichlet eigenvalues and Neumann eigenvalues is an interesting topic. We recall some related results on the Euclidean spaces. For a bounded Lipschitz domain $\Sigma \subset \mathbb{R}^n,$ denote by $\mu_k$ ($\lambda_k$ resp.) the $k$-th  eigenvalue of the Laplacian on $\Sigma$ with Neumann (Dirichlet resp.) boundary condition, ordered by 
	$$
	0 = \mu_1 < \mu_2 \leq \mu_3 \leq \cdots,\quad (0 < \lambda_1 < \lambda_2 \leq \lambda_3 \leq \cdots\mathrm{resp.}).
	$$

	By the classical variational principle, one immediately has $\mu_{k} \leq \lambda_k$ for all $k \in \mathbb{Z}_+$. In 1952, P\'olya \cite{P1952} proved that $\mu_2 < \lambda_1$ for any bounded Lipschitz domain in $\mathbb{R}^2.$ Later, Payne  \cite{P1955} proved that $\mu_{k+2} < \lambda_k$ for all $k \in \mathbb{Z}_+$ for a bounded convex domain with $C^2$ boundary in $\mathbb{R}^2.$ In 1991, Friedlander \cite{F1991} established that $\mu_{k+1} \leq \lambda_k$ for all $k \in \mathbb{Z}_+$ for any bounded $C^1$ domain in $\mathbb{R}^n,$ which is now called Friedlander's inequality. For more results in this direction, one refers to \cite{A1986,S1954,M1991,LW1986, HW2001, F2005, FL2009, AM2012,C2019,Filonov2004,GM2009,S2008,AL1997,H2008}.
	
	Recently, Rohleder \cite{R2023} introduce an innovative method based on the variational principle for vector fields on the plane, combined with Helmholtz decomposition, to prove that $\mu_{k+2} \leq \lambda_k, k \in \mathbb{Z}_+$ for any simply-connected and bounded Lipschitz domain in $\mathbb{R}^2$. This removes the convexity condition for domains in Payne's  result \cite{P1955}, and extends Friedlander's inequality.
	
    In this paper, we aim to extend Rohleder's method to general Riemannian surfaces and prove the comparison result for Dirichlet and Neumann eigenvalues. We require the following two ingredients: On one hand, Rohleder introduced a linear operator on vector fields defined on the planar domain, whose spectrum consists of Dirichlet and Neumann Laplacian eigenvalues. However, its geometric meaning is not clear for a general surface. Our key observation is that the spectrum of the Hodge Laplacian on $1$-forms, with the appropriate boundary conditions, serves the same purpose on surfaces. Using this, we can apply the standard variational principle and the Hodge decomposition to compare the Dirichlet and Neumann eigenvalues.
    
    On the other hand, we need suitable test functions for the variational arguments, which are provided by the following curvature condition. A $C^1$ function $f$ on a Riemannian surface $M$ is called a \emph{distance function} if it has unit gradient, i.e.  $|\nabla f|\equiv 1.$
	
	\begin{defi}\label{def:curvature}
		We say a domain $\Sigma$ in a Riemannian surface $M$ satisfies the \textit{curvature condition} if there exist an open set $U\supset \Sigma$ and a distance function with $f\in C^3(U)$ such that the Gaussian curvature $K$ satisfies
		\begin{align}
			K\le -|\operatorname{Hess} f|^2\quad\text{in } U.\label{eq:curvatureCondition}
		\end{align}
		For convenience, we also say that $f$ satisfies the curvature condition in $U$.
		\end{defi}
	One easily sees that the necessary condition for the above is the non-possitivity of the Gaussian curvature. Moreover, the planar domain satisfies the curvature condition by choosing the coordinate function as the distance function. For a smooth distance function $f,$ the sufficient and necessary condition for
	\eqref{eq:curvatureCondition} is 
	\begin{align}
		K(x)\le -h^2(x)\quad \forall x\in U,\label{eq:meanCurvature}
	\end{align}	
	where $h(x)$ is the curvature of the level set $f^{-1}(f(x)).$ This justifies the term ``curvature condition'', and relax the requirement in \cite{M1991} demanding for a constant mean curvature of level sets. Later, we will provide several interesting examples that satisfy the curvature condition, including warped product surfaces with log-convex warped functions, the flat cylinder, the hyperbolic plane, hyperbolic cusp, funnel, and etc. The following is the main result of the paper.
	
	\begin{thm}\label{thm:main}
		Let $\Sigma$ be a bounded Lipschitz domain in a Riemannian surface $M$ satisfying the curvature condition. Then
		\begin{align}
			\mu_{3-\beta_1} \leq \lambda_1, \label{eq:mainIneq} 
		\end{align} where $\beta_1$ is the first Betti number of $\Sigma.$ If $\Sigma$ is moreover smooth and simply connected, further we have the strict inequality
		\begin{align}
			\mu_{3} < \lambda_1. \label{eq:mainIneqStrict} 
		\end{align}
	\end{thm}
\begin{rem} The result is meaningful only when $\beta_1=0$ or $1,$ since it is trivial for $\beta_1\ge 2.$ For the case of $\R^2,$ it recovers the first inequality of Rohleder's result \cite{R2023} for a simply-connected bounded Lipschitz domain. This theorem generalizes part of the result to the surface case satisfying the curvature condition.
\end{rem}	
	 
The proof strategy of the theorem is as follows: First, for a smooth bounded domain, we introduce the Hodge Laplacian on $1$-forms with appropriate boundary conditions, whose spectrum consists of Dirichlet and Neumann Laplacian eigenvalues (see Proposition~\ref{prop:specDec}). On a surface, the Hodge Laplacian on $1$-forms decomposes into the upward and downward Laplacians. The spectrum of the upward (downward resp.) Laplacian coincides with that of the Hodge Laplacian on
$2$-forms ($0$-forms resp.), up to some null eigenvalues, and corresponds to the Dirichlet (Neumann resp.) Laplacian eigenvalues. This is a crucial property in dimension two, which is difficult to generalize to higher dimensions.
Second, we utilize Dirichlet eigenfunctions to construct specific $1$-forms, adopt suitable test functions provided by the curvature condition, and apply the variational principle for the Hodge Laplacian of $1$-forms to compare the Dirichlet and Neumann eigenvalues. Finally, we approximate the Lipschitz domain from the exterior by a sequence of smooth domains with a uniform cone condition. We prove the lower semi-continuity of the Neumann eigenvalues for the approximating sequence and extend the result to Lipschitz domains (see Appendix).

Theorem~\ref{thm:main} reduces this problem to finding an appropriate distance function on the surface that satisfies the curvature condition in Definition~\ref{def:curvature}. For applications, we discuss the curvature condition as follows. The first candidate is the Busemann function on a Hadamard surface, which is a simply-connected surface with nonpositive curvature. This function is a natural distance function to the infinity. For a space form of nonpositive curvature, such as the plane or the hyperbolic plane, the Busemann function satisfies the curvature condition; see Example~\ref{eg:spaceForm} (or \cite{CS2009}) for the hyperbolic case. In a general Hadamard surface, it remains an interesting question whether there exists a Busemann function that satisfies the curvature condition. See, e.g., \cite{HH1977, ISS2014, IKPS2017} for discussions of Busemann functions.
	
	The next theorem provides a criterion for the curvature condition in a twisted product surface.
	
	\begin{thm}\label{thm:logConvex}
		Let $M=(I\times S^1,g),$ for an interval $I\subset \R,$ be a twisted product surface with the metric
		$$
		g(r,\theta)=dr^2+\phi(r,\theta)^2 d\theta^2,\quad (r,\theta)\in M,$$ where $\phi$ is a positive smooth function.
		Then the radial distance function $f(r,\theta)=r$ satisfies the curvature condition if and only if $\phi(r,\theta)$ is log-convex in $r$, that is, $\partial^2_r\log\phi\ge 0$. In this case, every bounded Lipschitz (smooth resp.) domain $\Sigma\subset M$ satisfies the eigenvalue inequality \eqref{eq:mainIneq} (strict inequality \eqref{eq:mainIneqStrict} resp.).
	\end{thm}

	\begin{rem} When $\phi = \phi(r)$ does not depend on $\theta$, the twisted product structure simplifies to the well-known warped product. There are plenty of examples of warped product surfaces with log-convex warped functions. For example, one chooses $\phi=e^f$ for any convex function $f$ on $I.$ Typical examples are as follows:
		\begin{enumerate}
		\item Flat cylinder on $\R\times S^1$ with $\phi=\mathrm{const.};$
		\item Hyperbolic cusp on $\R\times S^1$ with $\phi(r)=e^r,$ see Example \ref{eg:cusp}.
		\item Hyperbolic collar (funnel resp.) on $\R\times S^1$ ($[0,+\infty)\times S^1$ resp.) with $\phi(r)=\cosh r,$ see Example \ref{eg:collar} for collars. In fact, $\phi$ satisfies the log-convex condition strictly, thus we allow a compact supported perturbation of the metric, i.e. 
		$$
		\phi=\cosh r +\eps h(r,\theta),\quad h\ge 0,\quad h\in C_c^\infty(\R\times S^1),
		$$
		where $\eps>0$ is sufficiently small and depends on the domain $\Sigma$.
		\end{enumerate}
	\end{rem}
	The fundamental structures in hyperbolic geometry consist of cusps, funnels, and collars. In a complete hyperbolic surface, the ends are either cusps or funnels, and there exists a neighborhood of simple closed geodesics forming a collar \cite{BMM2018,B1992}.
All of them satisfy the curvature conditions, and hence our results apply to hyperbolic geometry. We collect them in the following corollary.
	
	\begin{cor}\label{cor:main}
	Let $M$ be a space form with a nonpositive curvature, flat cylinder, hyperbolic cusp, funnel, or collar. For any bounded Lipschitz domain $\Sigma \subset M,$ the inequality \eqref{eq:mainIneq} holds.
	\end{cor}

	\begin{figure}[htbp]
		\includegraphics[width=0.45\textwidth]{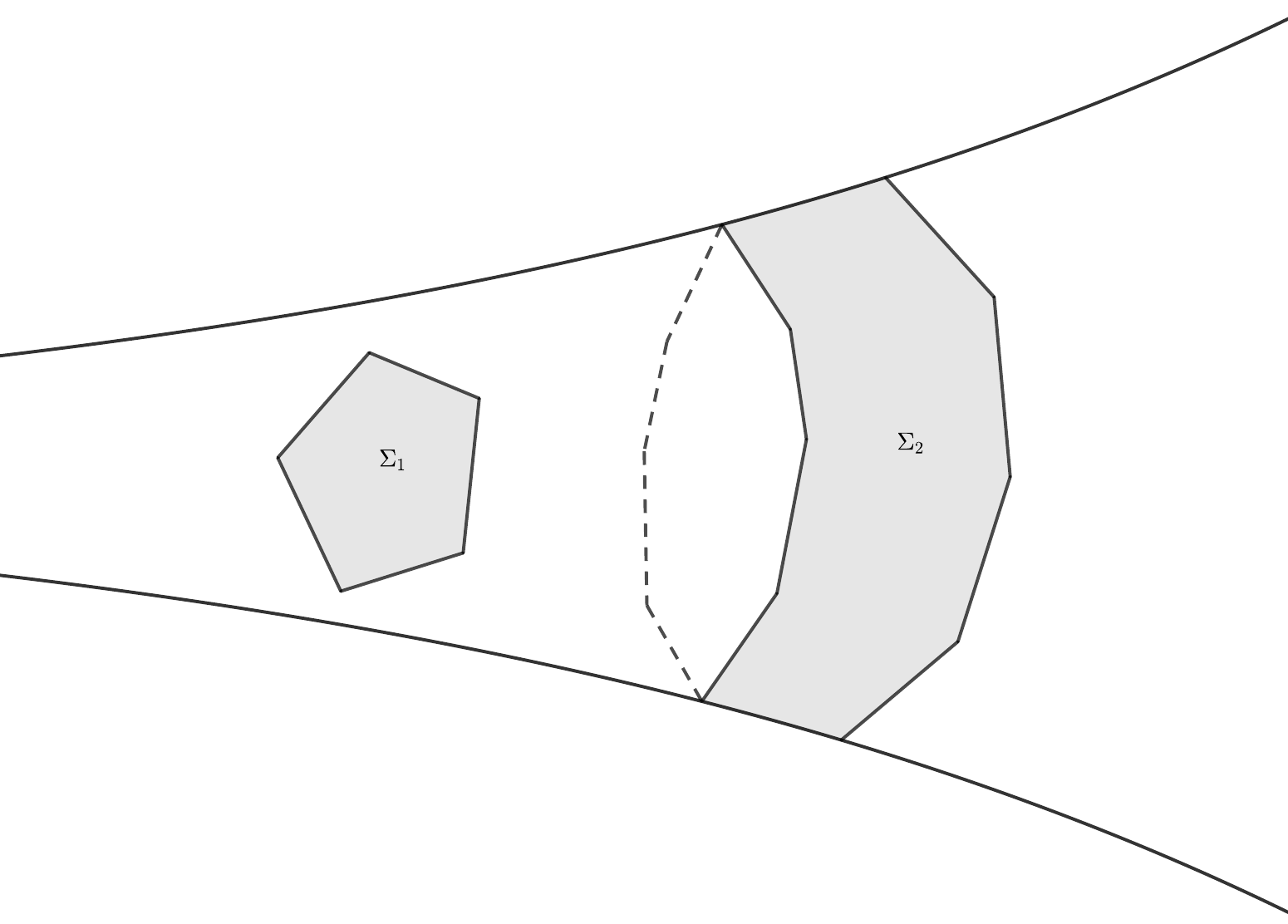}\quad 
		\includegraphics[width=0.45\textwidth]{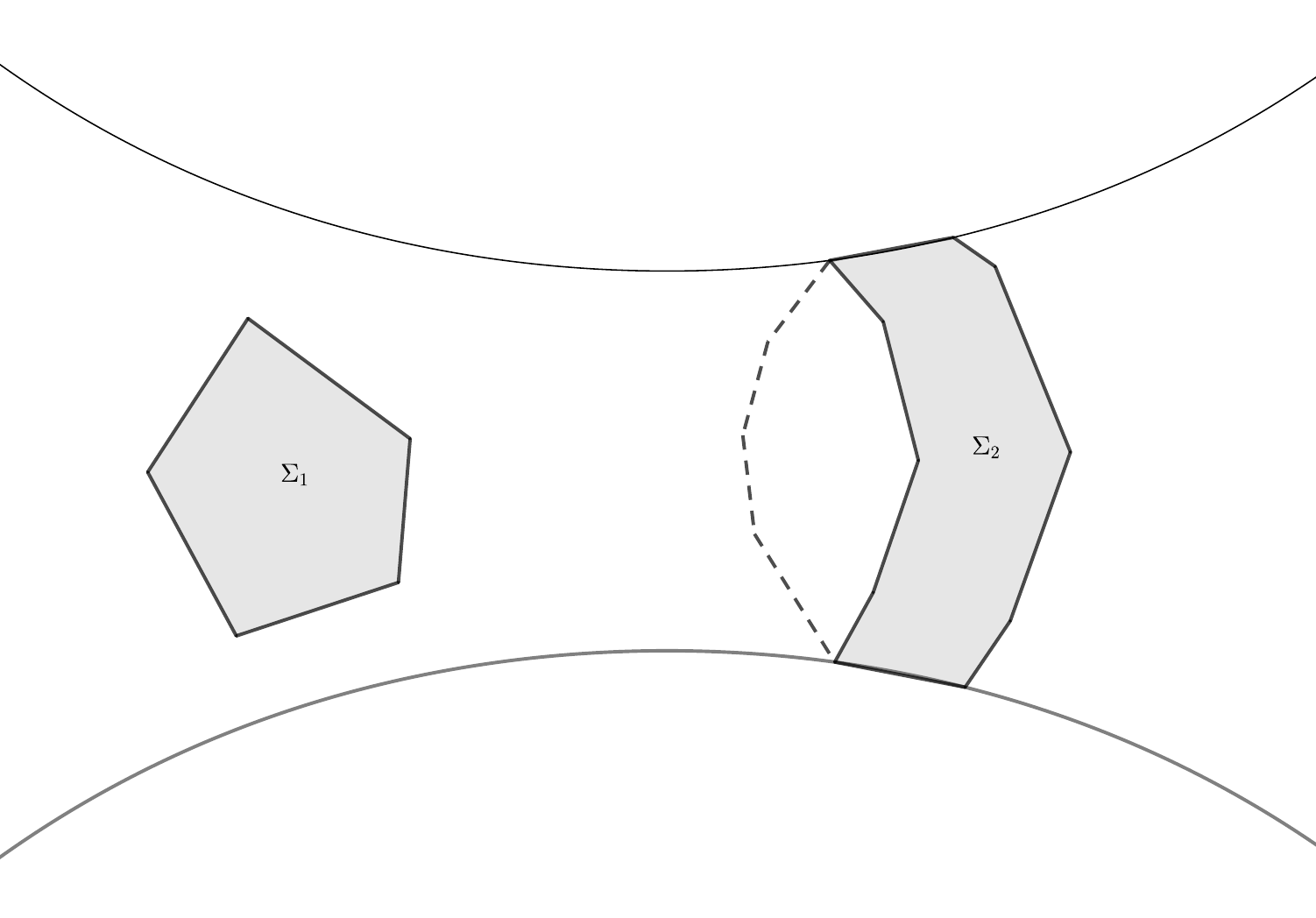}
		\caption{Bounded Lipschitz domains in a cusp and a collar. As $\Sigma_1$ is simply-connected, it follows that $\mu_{3}\le \lambda_1$. As for the annular domain $\Sigma_2,$ we obtain that $\mu_{2}\le \lambda_1.$}
	\end{figure}
	
	\begin{rem} Note that Mazzeo \cite{M1991} proved Friedlander's inequality $\mu_{k+1}\le \lambda_k$ on the hyperbolic plane, and our result improves his estimate for $k=1$ from $\mu_{2}$ to $\mu_{3}$ for simply-connected bounded domains. The cases of cylinders, cusps, funnels and collars are noteworthy and have not been investigated previously. Moreover, eigenvalue estimates for cusps, funnels and collars are novel contributions, since they are fundamental structures in hyperbolic surfaces.
	\end{rem}
	
% TODO: New minimal surfaces?
	
	It is also interesting to observe that the eigenvalue inequalities~\eqref{eq:mainIneq} and \eqref{eq:mainIneqStrict} hold locally on certain minimal surfaces in $\mathbb{R}^3$, including the catenoid and helicoid, which are among the most common and useful minimal surfaces \cite{CM2011}. This part is shown in Examples~\ref{eg:helicoid} and \ref{eg:catenoid}.
	
	\begin{figure}[htbp]
		\includegraphics[width=0.45\textwidth]{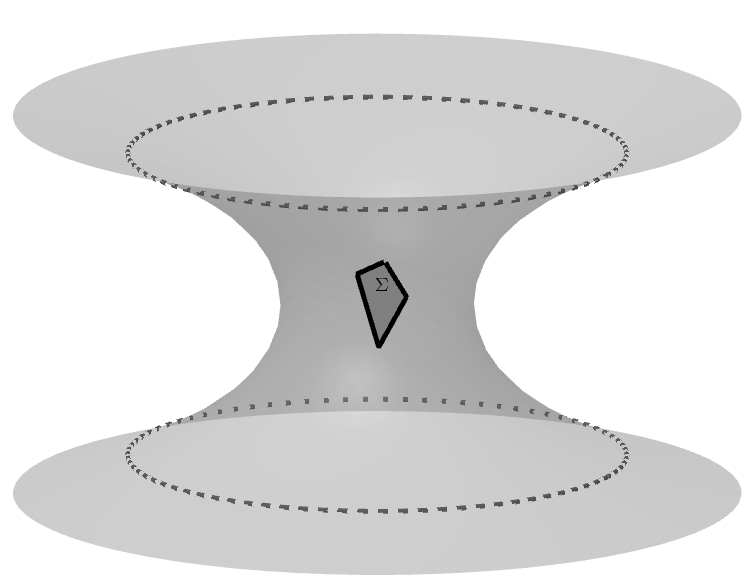}\quad 
		\includegraphics[width=0.45\textwidth]{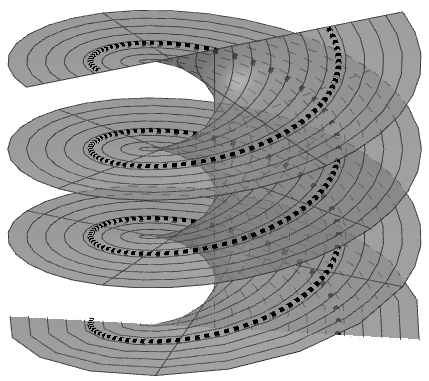}
		\caption{A catenoid and a helicoid. A domain $\Sigma$ compactly contained in the region bounded by the dotted lines satisfies the curvature condition.}
	\end{figure}
	
	The paper is organized as follows: we introduce basic concepts of differential forms and Hadamard manifolds in Section~\ref{sec:Pre}. Section~\ref{sec:SobolevForm} contains the theory of Sobolev spaces for differential forms based on \cite{S1995}. Section~\ref{sec:surface} is devoted to necessary analytical tools: spectral decomposition for the Hodge Laplacian, Dirichlet form, and variational principle. In Section~\ref{sec:Friedlander}, we prove the main theorems and provide several examples in detail. The smooth approximation of a Lipschitz domain and the semi-continuity of Neumann eigenvalues is included in the Appendix.

\section{Preliminaries}\label{sec:Pre}

\subsection{Notations for Differential Forms}
	Let $(M^n,g)$ be an oriented $n$-dimensional Riemannian manifold. We consider the exterior $k$-form bundle $\Lambda^k(M)$ and its smooth sections $\Omega^k(M) = \Gamma(\Lambda^k(M))$, which constitute the space of differential forms of degree $k$ on $M$, where $k \in \mathbb{N}$.
	
	Using the musical isomorphism $\sharp: T^*_pM \rightarrow T_pM$ defined by
	$$
	g(\omega^{\sharp}, v) = \omega(v), \quad \forall v \in T_pM,
	$$
	we can define the pointwise inner product on $\Lambda^1(M)$ as
	$$
	\langle \omega, \nu \rangle = g(\omega^\sharp, \nu^\sharp),
	$$
	which extends to $\Lambda^k(M)$ by
	$$
	\langle \omega_1 \wedge \cdots \wedge \omega_k, \nu_1 \wedge \cdots \wedge \nu_k \rangle = \det\{\langle \omega_i, \nu_j \rangle \}_{i,j=1}^k.
	$$
	The inverse map of $\sharp$ is written as $\flat: T_pM \rightarrow T^*_pM$, and both maps naturally extend to isomorphisms between $TM$ and $T^*M$.
	
	We denote the volume form by $\operatorname{vol}_M \in \Omega^n(M)$, which satisfies
	$$
	\operatorname{vol}_M(X_1, \ldots, X_n) = \sqrt{\det\{g(X_i, X_j)\}_{i,j=1}^n}.
	$$
	The Hodge operator $*: \Omega^k(M) \rightarrow \Omega^{n-k}(M)$ is defined as
	$$
	\omega \wedge (*\nu) = \langle \omega, \nu \rangle \operatorname{vol}_M, \quad \forall \omega \in \Omega^k(M).
	$$
	
	\begin{prop}[\cite{L2002}]
		The Hodge operator satisfies the following properties for all $\omega, \nu \in \Omega^k(M)$:
		\begin{enumerate}[(i)]
			\item $*(*\omega) = (-1)^{k(n-k)} \omega,$
			\item $\langle *\omega, *\nu \rangle = \langle \omega, \nu \rangle.$
		\end{enumerate}
	\end{prop}

	Suppose $\Sigma \subset M$ is a compact $n$-dimensional submanifold with boundary, equipped with an induced atlas, also known as the $\partial$-manifold in \cite{S1995}. Denote by $\Omega^*(\Sigma)$ the set of differential forms on $\Sigma$. Let $\iota: \Sigma \rightarrow M$ be the inclusion map. Then we have $\iota^* \Omega^*(M) \subset \Omega^*(\Sigma)$.
	
	The $L^2$ inner product on $\Omega^k(\Sigma)$ is defined as
	$$
	(\omega_1, \omega_2) := \int_\Sigma \langle \omega_1, \omega_2 \rangle \operatorname{vol}_M = \int_\Sigma \omega_1 \wedge (*\omega_2), \quad \omega_1, \omega_2 \in \Omega^k(\Sigma).
	$$ 
	
	The co-differential operator is given by $d^* \omega := (-1)^{nk+n+1} *d(*\omega)$ for $\omega \in \Omega^k(\Sigma)$. By Stokes' theorem, we have the following lemma.
	
	\begin{lemma}[Green's Formula]\label{lem:Green}
		Let $i: \partial \Sigma \rightarrow \Sigma$ be the inclusion map. Then, for $\omega_1 \in \Omega^{k-1}(\Sigma)$ and $\omega_2 \in \Omega^k(\Sigma)$, we have
		\begin{align}
			(d \omega_1, \omega_2) = (\omega_1, d^* \omega_2) + \int_{\partial \Sigma} i^*(\omega_1 \wedge *\omega_2). \label{eq:Green}
		\end{align}
	\end{lemma}
	
	\begin{proof}
		Let $\omega = \omega_1 \wedge *\omega_2$. Then,
		$$
		\begin{aligned}
			d\omega &= d\omega_1 \wedge *\omega_2 + (-1)^{k-1} \omega_1 \wedge (d *\omega_2) \\
			&= d\omega_1 \wedge *\omega_2 + (-1)^{nk+n+k} *\omega_1 \wedge d^* \omega_2 \\
			&= d\omega_1 \wedge *\omega_2 + (-1)^{nk+n+k+(k-1)(n-k+1)} d^* \omega_2 \wedge *\omega_1 \\
			&= d\omega_1 \wedge *\omega_2 - d^* \omega_2 \wedge *\omega_1.
		\end{aligned}
		$$
		Thus, by Stokes' theorem, we obtain
		$$
		\int_{\partial \Sigma} i^* \omega = \int_\Sigma d\omega = \int_\Sigma d\omega_1 \wedge *\omega_2 - \int_\Sigma d^* \omega_2 \wedge *\omega_1 = (d\omega_1, \omega_2) - (d^* \omega_2, \omega_1).
		$$
		
		This proves the result.
	\end{proof}

	\begin{example}
		For $M = \mathbb{R}^2$, one easily sees that for a vector field $v = u_1 e_1 + u_2 e_2 \in \Gamma(T\Sigma)$, we have $v^\flat = u_1 dx_1 + u_2 dx_2 \in \Omega^1(\Sigma)$. Then,
		$$
		d (v^\flat) = (\partial_1 u_2 - \partial_2 u_1) \operatorname{vol}_M = (\operatorname{curl} v) \operatorname{vol}_M.
		$$
		For the Hodge operator, we have
		$$
		*dx_1 = dx_2, \quad *dx_2 = -dx_1, \quad *(v^\flat) = u_1 dx_2 - u_2 dx_1 = (v^\perp )^\flat,
		$$
		where $v^\perp := -u_2 e_1 + u_1 e_2$ is the vector $v$ rotated by $\frac{\pi}{2}$ counterclockwise. Then we obtain
		$$
		-d^* (v^\flat) = *d(u_1 dx_2 - u_2 dx_1) = \partial_1 u_1 + \partial_2 u_2 = \operatorname{div} v.
		$$
		
	\end{example}
	
	The Hodge Laplacian is defined as $\Delta_k := d^*d + dd^*$ on $k$-forms. Naturally, we have $d^*|_{\Omega^0(\Sigma)} = 0$ and $d|_{\Omega^n(\Sigma)} = 0$. We may omit the subscript $k$ when the degree is clear from the context.

\subsection{Busemann Function on a Hadamard Manifold}
	Suppose $H$ is a Hadamard manifold, i.e. a simply-connected, complete Riemannian manifold with non-positive sectional curvature. By the Cartan-Hadamard theorem, $H$ is diffeomorphic to $\mathbb{R}^n$.
	
	Let $\gamma_1$ and $\gamma_2$ be two geodesic rays, $\gamma_i: [0, \infty) \rightarrow H$, $i=1,2$. We always consider geodesics with unit speed. The geodesics $\gamma_1$ and $\gamma_2$ are said to be \textit{asymptotically equivalent} if there exists a constant $C>0$ such that $d(\gamma_1(t), \gamma_2(t)) < C$ for all $t \ge 0$. The set of all asymptotically equivalent classes forms the \textit{ideal boundary} of $H$, denoted by $\partial H$. This boundary can be regarded as the ``point at infinity'' of $H$, meaning that a geodesic $\gamma$ is said to ``converge'' to $\theta \in \partial H$ at infinity if $\gamma$ represents the class $\theta$.
	
	Fix an arbitrary point $o \in H$ and let $\theta \in \partial H$. We can find a unique geodesic ray such that $\gamma(0) = o$ and $[\gamma] = \theta$. The corresponding Busemann function is defined as
	$$
	B_\theta(x) := \lim_{t \rightarrow \infty} \left(d(x, \gamma(t)) - t\right).
	$$
	
	\begin{prop}[\cite{ISS2014}]
		The Busemann function $B_\theta$ satisfies the following properties:
		\begin{enumerate}[(i)]
			\item $B_\theta$ is a convex function with $C^2$ regularity.
			\item $|dB_\theta| \equiv 1$.
		\end{enumerate}
	\end{prop}

	The level set of $B_\theta$ containing $x \in H$ is called a horosphere centered at $\theta$ passing through $x$, denoted by $\mathcal{H}_{(x,\theta)}$, i.e.,
	$$
	\mathcal{H}_{(x,\theta)} = \{y \in H \mid B_{\theta}(y) = B_{\theta}(x)\}.
	$$
	By the implicit function theorem, the horosphere $\mathcal{H}_{(x,\theta)}$ has $C^2$ regularity.
	
	The gradient field $\nabla B_{\theta}$, when restricted to $\mathcal{H}_{(x,\theta)}$, yields the $C^1$ unit vector field $N$ that is outward normal to $\mathcal{H}_{(x,\theta)}$. The second fundamental form is given by the Hessian
	$$
	h(v, w) = \langle -\nabla_v N, w \rangle = -\nabla^2 B_\theta(v, w), \quad v, w \in T_y \mathcal{H}_{(x,\theta)}.
	$$
	This form is negative semi-definite because $B_\theta$ is convex.

\section{Sobolev Space for Differential Forms}\label{sec:SobolevForm}

\subsection{General Case}
For a $\partial$-manifold $\Sigma\subset M,$ the form space $L^2\Omega^k(\Sigma)$, $0\le k\le n$, is defined as the completion of $\Omega^k(\Sigma)$ with respect to the induced norm
$$
\|\omega\|_{L^2} := \sqrt{(\omega, \omega)}.
$$
Moreover, the Sobolev space $H^s\Omega^k(\Sigma)$ is defined as the completion of $\Omega^k(\Sigma)$ with respect to the norm
\begin{align}
	\|\omega\|_{H^s} := \sqrt{\sum_{r=0}^s \|\nabla^r \omega\|_{L^2}^2}, \quad s \ge 0, \label{eq:Hsnorm}
\end{align}
where $\nabla$ is the Levi-Civita connection on $M$. Similarly, one can define $W^{s,p}\Omega^k(\Sigma)$ generally; see \cite{S1995}. For $s < 0$, the corresponding Sobolev space is defined as the dual space of $H^{-s}.$

It is clear that these spaces are Hilbert spaces with the corresponding inner product. Since $\|\omega\|_{L^2} = \|*\omega\|_{L^2}$ for all $\omega \in \Omega^k(\Sigma)$, the Hodge star operator has a continuous extension
$$
* : L^2\Omega^k(\Sigma) \rightarrow L^2\Omega^{n-k}(\Sigma),
$$
and it is an isometry. Given that $*(\nabla_Y \omega) = \nabla_Y(*\omega)$ for all $Y \in T\Sigma$ and $\omega \in \Omega^k(\Sigma)$, we also have $\|\omega\|_{H^s} = \|*\omega\|_{H^s}$, and thus the Hodge star operator extends to an isometry between Sobolev spaces as well.

By the standard theory of Sobolev spaces \cite{S1995}, regarding the trace, we have
$$
\|i^*\omega\|_{H^s\Omega^k(\partial \Sigma)} \le \|\omega|_{\partial \Sigma}\|_{H^s\Gamma(\Lambda^k(\partial \Sigma))} \le C\|\omega\|_{H^{s+1}\Omega^k(\Sigma)},
$$
from which we obtain the continuous (and compact) linear map
$$
i^* : H^{s+1}\Omega^k(\Sigma) \rightarrow H^s\Omega^k(\partial \Sigma), \quad s \ge 0.
$$

The differential operators $d$ and $d^*$ continuously extend to the Sobolev space such that
$$
d: H^{s+1}\Omega^k(\Sigma) \rightarrow H^s\Omega^{k+1}(\Sigma), \quad d^*: H^{s+1}\Omega^k(\Sigma) \rightarrow H^s\Omega^{k-1}(\Sigma), \quad s \ge 0,
$$
since both can be expressed using $\nabla$ after choosing a coordinate system.

By the continuity of the exterior derivative and the trace formula, Stokes' theorem holds for Sobolev spaces as well, allowing us to extend Green's formula~\eqref{eq:Green} to Sobolev spaces. Specifically, we have
\begin{align}
	(d \omega_1, \omega_2) = (\omega_1, d^* \omega_2) + \int_{\partial \Sigma} i^*(\omega_1 \wedge *\omega_2), \quad \forall \omega_1\in H^1\Omega^{k-1}(\Sigma),\:\omega_2\in H^1\Omega^k(\Sigma).\label{eq:SobolevGreen}
\end{align}

We define the following form space for $s \ge 1$:
\begin{align}
	H^s\Omega_D^k(\Sigma) = \{\omega \in H^s\Omega^k(\Sigma) \mid i^*(*\omega) = 0\}. \label{eq:DSobolev}
\end{align}
As we will see later, this condition represents the absence of a normal component for forms at the boundary. The subscript $D$ represents exactly the Dirichlet condition.

Using standard techniques for elliptic equations, we can derive the following Hodge decomposition on the $\partial$-manifold $\Sigma$, which will be crucial in our paper.

\begin{thm}[\cite{S1995}, Corollary~2.4.9]\label{thm:mainDecomposition}
	Let $\Sigma$ be a $\partial$-manifold. We have the following $L^2$-orthogonal decomposition:
	\begin{align}
		L^2\Omega^k(\Sigma) = dH^1\Omega^{k-1}(\Sigma) \oplus d^*H^1\Omega^{k+1}_D(\Sigma)\oplus\HH^k_D(\Sigma),\label{eq:mainDec}
	\end{align}
	where $\HH^k_D(\Sigma)$ is the harmonic form with Dirichlet boundary defined as
	$$
	\HH^k_D(\Sigma):=\{\omega\in H^1\Omega^k_D(\Sigma)\mid d\omega=0,\,d^*\omega=0\}.
	$$
\end{thm}

Moreover, we have the following Hodge isomorphism.
\begin{thm}[\cite{S1995}, Theorem~2.6.1]\label{thm:isomorphism}
	For the $k$-th cohomology group $H^k_{dR}(\Sigma)$, we have
	$$
	H^k_{dR}(\Sigma)\cong \HH^k_D(\Sigma).
	$$
\end{thm}

When $\Sigma$ is simply-connected, we have $\HH^k_D(\Sigma)=0,$ and in \eqref{eq:mainDec} we only have the leading two spaces. This case is also known as the Helmholtz decomposition when we are dealing with vector fields on the plane.

\subsection{Surface Case}
In this subsection, we will primarily focus on the case when $n=2$. Consider the space
$$
\EEE(\Sigma) := \{\omega \in L^2\Omega^1(\Sigma) \mid d^*\omega \in L^2\Omega^0(\Sigma) = L^2(\Sigma)\}
$$
with the corresponding norm
$$
\|\omega\|^2_\EEE := \|\omega\|^2_{L^2} + \|d^*\omega\|^2_{L^2}.
$$
This is a Hilbert space, and $\Omega^1(\Sigma)$ is dense in $\EEE(\Sigma)$ under this norm.

Let $\nu$ be the outer normal vector of $\partial \Sigma$. The mapping
$$
\Omega^1(\Sigma) \ni \omega \mapsto \omega|_{\partial \Sigma}(\nu)
$$
extends to the map $\EEE(\Sigma) \rightarrow H^{-\frac{1}{2}}(\partial \Sigma)$; see \cite[Chapter XIX, \S1, Theorem 2]{DL2000} for the case on the plane $\R^2$. The situation is similar for a manifold, particularly for Hadamard surface, since there is a direct diffeomorphism to the plane, and when we focus on a compact domain.

In fact, we have 
$$
i^*(*\omega) = \omega|_{\partial\Sigma}(\nu) \operatorname{vol}_{\partial\Sigma},
$$
so that $i^*(*\omega)$ can be used to express the boundary condition as in \eqref{eq:DSobolev}. Notably, Green's formula \eqref{eq:Green} extends to $\EEE(\Sigma)$ as follows
\begin{align}
	(d u, \omega) = (u, d^* \omega) + \int_{\partial \Sigma} i^*(*\omega) u, \quad \forall u \in H^1(\Sigma), \, \omega \in \EEE(\Sigma). \label{eq:EEEGreen}
\end{align}

Similarly, for all $\omega \in *\EEE(\Sigma)$, we have 
$$
d\omega = -*(d^*(*\omega)) \in *L^2(\Sigma) = L^2\Omega^2(\Sigma),
$$
and therefore
\begin{align}
	(d \omega, A) = (\omega, d^* A) + \int_{\partial \Sigma} (*A)(i^*\omega), \quad \forall A \in H^1\Omega^2(\Sigma) = * H^1(\Sigma), \omega \in *\EEE(\Sigma). \label{eq:EEEGreen'}
\end{align}

Let $u \in H^1(\Sigma)$ such that $\Delta u = d^*du \in L^2(\Sigma)$ in the weak sense. Then $\omega = du \in \EEE(\Sigma)$, and hence
$$
\partial_\nu u = du(\nu) = (*_{\partial\Sigma}) i^*(*du) \in H^{-\frac{1}{2}}(\partial \Sigma)
$$
is well-defined. We will use this to define the Neumann boundary condition.

In the setting of Theorem~\ref{thm:mainDecomposition}, for the surface case with $H^1_{dR}(\Sigma)=0$ there is a well-known Helmholtz decomposition
$$
L^2\Omega^1(\Sigma) = dH^1(\Sigma) \oplus d^* H^1\Omega^2_D(\Sigma) = dH^1(\Sigma) \oplus *dH^1_0(\Sigma).
$$

\section{Eigenvalue Problems on Riemannian Surfaces}\label{sec:surface}

\subsection{Dirichlet and Neumann Problems}
Consider the case when $n=2$, i.e., $M$ is a Riemannian surface, and $\Sigma$ is a bounded smooth domain on $M$. The domain of the Dirichlet problem is defined as
$$
\dom (\Delta_D) = \{A \in H^1\Omega^2_D(\Sigma) \mid \Delta A = dd^*A \in L^2\Omega^2(\Sigma)\},
$$
and the domain of the Neumann problem is
$$
\dom (\Delta_N) = \{u \in H^1(\Sigma) \mid \Delta u = d^*du \in L^2(\Sigma), \ \partial_\nu u|_{\partial \Sigma} = 0\}.
$$
Here we note that $*\dom(\Delta_D) \subset H_0^1(\Sigma)$ corresponds to the domain of the classical Dirichlet problem, and $\partial_\nu u|_{\partial\Sigma}$ is well-defined because $du \in \EEE(\Sigma).$ The Neumann boundary condition is equivalent to $i^*(*du) = 0.$

Consider the Dirichlet eigen-form $\Phi \in \dom (\Delta_D)$ such that $\Delta_D \Phi = dd^*\Phi = \lambda \Phi$. Let $\Phi = \varphi \: \operatorname{vol}_M$, where $\varphi \in H^1_0(\Sigma)$ satisfies 
$$
\Delta \varphi = d^*d\varphi = -*d*d(**\varphi) = *dd^*\Phi = \lambda \varphi,
$$
which is the classical eigenfunction for the Dirichlet problem on $\Sigma$.

Note that for $\varphi \in *\dom(\Delta_D) \subset H_0^1(\Sigma)$, we have $d\varphi \in \EEE(\Sigma)$, and then 
\begin{align}
	i^*(d\varphi) = d(i^*\varphi) = 0.\label{eq:DirichletCondition}
\end{align}
Here, we use the continuity of $i^*$ and $d$, as well as their commutativity for smooth differential forms. For a detailed proof, see \cite[Lemma~2.2]{R2023}.

The Neumann eigenfunction is the classical one: $\psi \in \dom(\Delta_N)$ such that $\Delta_N \psi = d^*d\psi = \mu \psi$. Denote the Neumann eigenvalues by 
$$
0 = \mu_1 < \mu_2 \le \mu_3 \le \cdots
$$
with multiplicities, and let
$$
0 < \lambda_1 < \lambda_2 \le \lambda_3 \le \cdots
$$
be the eigenvalues of $\Delta_D$, also with multiplicities.

Take orthonormal eigenbases $(\psi_1, \psi_2, \cdots)$ and $(\varphi_1, \varphi_2, \cdots)$ of $L^2(\Sigma)$ such that $\Delta_N\psi_i = \mu_i\psi_i$ and $\Delta_D \varphi_i = \lambda_i\varphi_i$ hold for all $i \in \mathbb{Z}_+.$ Then $\{\Phi_i := \varphi_i \operatorname{vol}_M\}$ form orthonormal eigenbases of $L^2\Omega^2(\Sigma)$ with $\Delta_D \Phi_i = \lambda_i \Phi_i$.

\subsection{Spectral Decomposition and Dirichlet Form} \label{sec:Dirichlet}

\begin{defi}
	We define the Dirichlet form $\alpha$ on $H^1\Omega^k_D(\Sigma)$ by
	$$
	\alpha[\omega_1, \omega_2] := (d\omega_1, d\omega_2) + (d^*\omega_1, d^*\omega_2),
	$$
	and we write $\alpha[\omega] := \alpha[\omega, \omega]$ for simplicity.
\end{defi}

Clearly, $\alpha$ is bounded since $d$ and $d^*$ are continuous in $H^1\Omega^k_D(\Sigma)$. It is positive semi-definite with the kernel being $\HH^k(\Sigma)\cong H^k_{dR}(\Sigma)$ by Theorem~\ref{thm:isomorphism}.

We define the following inner product on $H^1\Omega^k_D(\Sigma)$
$$
(\omega_1, \omega_2)_\alpha := \alpha[\omega_1, \omega_2] + (\omega_1, \omega_2).
$$
By Gaffney's inequality \cite[Corollary~2.1.6]{S1995}, the norm induced by $(-,-)_\alpha$ is equivalent to the canonical $H^1$ norm \eqref{eq:Hsnorm} on the space $H^1_D\Omega^k(M)$, and we denote it by $\| - \|_\alpha$.

Let $A: H^1\Omega^k_D(\Sigma) \to H^{-1}\Omega^k(\Sigma)$ be the bounded $L^2$-self-adjoint operator such that
$$
\alpha[u, v] = (Au, v), \quad \forall u, v \in H^1\Omega^k_D(\Sigma).
$$
In fact, by Green's formula \eqref{eq:SobolevGreen}, we have $Au = \Delta u$ for $u \in H^2\Omega^k_D(\Sigma)$ with $i^*(*du) = 0$.

From now on, we will focus on the case for $n = 2$ and $k = 0, 1, 2$.

\begin{prop}\label{prop:specDec}
	Let $\Sigma$ be a bounded smooth domain on a smooth surface. The $1$-forms $\left\{\frac{1}{\sqrt{\mu_i}}d\psi_i\right\}_{i=2}^\infty$ and $\left\{\frac{1}{\sqrt{\lambda_i}}d^*\Phi_i\right\}_{i=1}^\infty$ form orthonormal eigen-bases of $dH^1(\Sigma)$ and $d^*H^1\Omega_D^2(\Sigma)$, respectively. These eigen-bases satisfy
	$$
	A(d\psi_i) = \mu_i d\psi_i, \quad A(d^*\Phi_i) = \lambda_i d^*\Phi_i.
	$$
	Let $\Spec_+$ be the positive spectrum. Then
	\begin{align}
		\Spec_+ A = \Spec_+ \Delta_D \cup \Spec_+ \Delta_N .\label{eq:spectrumDec}
	\end{align}
\end{prop}

\begin{proof}
	For $i \ge 2$, $d\psi_i \in H^1\Omega_D^1(\Sigma)$, since $d(d\psi_i) = 0$, $d^*(d\psi_i) = \mu_i\psi_i \in L^2(\Sigma)$, and $i^*(*d\psi_i) = 0$ by the Neumann boundary condition. Similarly, $d^*\Phi_i \in H^1\Omega_D^1(\Sigma)$ for $i \ge 1$, and the boundary condition holds by the Dirichlet condition as in \eqref{eq:DirichletCondition},
	$$
	i^*(*d^*\Phi_i) = -i^*(**d*\Phi_i) = -i^*(d\varphi_i) = 0.
	$$
	
	The orthogonality follows from Green's formula:
	$$
	(d^*\Phi_i, d^*\Phi_j) = (dd^*\Phi_i, \Phi_j) = \lambda_i (\Phi_i, \Phi_j), \quad (d\psi_i, d\psi_j) = (d^*d\psi_i, \psi_j) = \mu_i (\psi_i, \psi_j).
	$$
	
	Suppose that there exists a function $\psi \in H^1(\Sigma)$ such that $d\psi \perp d\psi_i$ for all $i \ge 2$. Then
	$$
	0 = (d\psi, d\psi_i) = \mu_i (\psi, \psi_i), \quad \forall i \ge 2,
	$$
	which implies that $\psi = c\psi_1$ is a constant as $\{\psi_i\}_{i=1}^\infty$ is an orthonormal basis of $L^2(\Sigma)$. Thus, $d\psi = 0$, showing that $\{\frac{1}{\sqrt{\mu_i}}d\psi_i\}_{i=2}^\infty$ is an orthonormal basis of $dH^1(\Sigma)$. Similarly, $\{\frac{1}{\sqrt{\lambda_i}}d^*\Phi_i\}_{i=1}^\infty$ is an orthonormal basis of $d^*H^1\Omega_D^2(\Sigma)$.
	
	Finally, they are eigenforms of $A$ because for any $v \in H^1\Omega_D^1(\Sigma)$,
	$$
	(Ad^*\Phi_i, v) = \alpha[d^*\Phi_i, v] = (dd^*\Phi_i, dv) = \lambda_i (\Phi_i, dv) = \lambda_i (d^*\Phi_i, v),
	$$
	$$
	(Ad\psi_i, v) = \alpha[d\psi_i, v] = (d^*d\psi_i, d^*v) = \mu_i (\psi_i, d^*v) = \mu_i (d\psi_i, v).
	$$
	Combined with the Hodge decomposition in Theorem~\ref{thm:mainDecomposition}, this proves the positive spectral decomposition \eqref{eq:spectrumDec}.
\end{proof}

\begin{rem}\label{rem:general}
	In fact, one can generalize the results by replacing $\Delta_D$ with the downward Laplacian $L_k = d_{k-1}d^*_k$ and $\Delta_N$ with the upward Laplacian $U_k = d^*_{k+1}d_k$. This leads to the positive spectrum decomposition for the Laplacian $\Delta_k$ acting on general Sobolev $k$-forms
	$$
	\Spec_+\left(\Delta_k\right) = \Spec_+\left(U_{k-1}\right) \cup \Spec_+\left(L_{k+1}\right).
	$$
	
	The proof follows from a similar method presented before.
\end{rem}

Finally, consider the Rayleigh quotient defined by $\eta(\omega) = \frac{\alpha[\omega]}{\|\omega\|_{L^2}^2}$. Using standard variational methods, we obtain the following result:

\begin{thm}
	Let $\eta_1 \le \eta_2 \le \cdots$ be the eigenvalues of the operator $A$. Then,
	$$
	\eta_k = \min_{\begin{subarray}{c}
			E \subset H^1\Omega_D^1(\Sigma) \\
			\dim E = k
	\end{subarray}} \max_{\omega \in E \setminus \{0\}} \eta(\omega), \quad \forall k \in \mathbb{Z}_+.
	$$
\end{thm}

\section{Proof of main results}\label{sec:Friedlander}
	\subsection{Proof of the main theorem}
	In this section, we prove Theorem~\ref{thm:main}. Let $\Sigma$ satisfy the curvature condition in Definition~\ref{def:curvature} with the distance function $f$. We define $\nu:=df\in H^1\Omega^1(\Sigma)$. Consequently, $d\nu=0$, $d^*\nu=\Delta f$, and $|\nu|=|\nabla f|\equiv 1$.
	
	Let $(\Phi_i=\varphi_i \vol_M,\lambda_i)$ be the Dirichlet eigenpair of $\Delta_D$. We obtain
	$$
	\|\varphi_i\nu\|_{L^2}^2=\int_{\Sigma}\varphi_i^2|\nu|^2 \, d\vol_M=\|\varphi_i\|_{L^2}^2=1.
	$$
	
	\begin{lemma}\label{lem:main-lem}
		Consider a bounded smooth domain $\Sigma$ on a surface $M$ that satisfies the curvature condition. For Dirichlet eigenfunctions $\{\varphi_i\}_{i=1}^\infty$ on $\Sigma$ and $\nu$ as defined previously, the following inequality holds
		$$
		\alpha[\varphi_i\nu], \alpha[\varphi_i(*\nu)]\le \lambda_i, \quad \forall i\in \Z_+.
		$$
		Furthermore, the cross term vanishes:
		$$
		\alpha[\varphi_i\nu,\varphi_i(*\nu)]=0.
		$$
	\end{lemma}
	
	\begin{proof}
		For the Dirichlet form $\alpha$, we have
		\begin{align}
			\alpha[\varphi_i\nu] &= (d(\varphi_i\nu),d(\varphi_i\nu)) + (d^*(\varphi_i\nu),d^*(\varphi_i\nu)) \nonumber\\
			&= (d\varphi_i\wedge \nu, d\varphi_i\wedge \nu) + (d\varphi_i\wedge *\nu + \varphi_i d*\nu, d\varphi_i\wedge *\nu + \varphi_i d*\nu) \nonumber\\
			&= \|d\varphi_i\|_{L^2}^2 + 2(d\varphi_i\wedge *\nu, \varphi_i d*\nu) + \|\varphi_i d^*\nu\|_{L^2}^2 \label{eq:third}\\
			&= (d^*d\varphi_i,\varphi_i) + (2\varphi_id\varphi_i\wedge *\nu, (**)d*\nu) + \|\varphi_i d^*\nu\|_{L^2}^2 \nonumber \\
			&= \lambda_i \|\varphi_i\|_{L^2}^2 - (d\varphi_i^2\wedge *\nu, *d^*\nu) + \|\varphi_i \Delta f\|_{L^2}^2.\label{eq:last}
		\end{align}
		The equality \eqref{eq:third} holds because for $d\varphi_i = a \nu + b (*\nu)$ with $a,b\in \R$,
		$$
		\|d\varphi_i\wedge \nu\|^2_{L^2} = b^2, \quad \|d\varphi_i\wedge *\nu\|^2_{L^2} = a^2, \quad \|d\varphi_i\|^2 = a^2 + b^2.
		$$
		For the middle term in \eqref{eq:last}, we have
		$$
		\begin{aligned}
			(d\varphi_i^2\wedge *\nu, *d^*\nu) &= \left(\tri{d\varphi_i^2,\nu}\vol_M, \Delta f \vol_M\right) \\
			&= (d\varphi_i^2, (\Delta f) \nu) \\
			&= (\varphi_i^2, d^*((\Delta f) \nu)) \quad \text{(because } \varphi_i^2|_{\partial\Sigma} = 0\text{)}.
		\end{aligned}
		$$
		Provided that the term $d^*(\Delta f df) \ge (\Delta f)^2$ everywhere in $\Sigma$, we have 
		$$
		(d\varphi_i^2\wedge *\nu, *d^*\nu)=(\varphi_i^2, d^*((\Delta f) \nu))\ge (\varphi_i^2, (\Delta f)^2) =\|\varphi_i \Delta f\|_{L^2}^2,
		$$
		and consequently in $\eqref{eq:last}$ we derive
		$$\alpha[\varphi_i\nu] \le \lambda_i \|\varphi_i\|^2_{L^2} = \lambda_i.$$
		
		For the term $d^*(\Delta f df)$, we have
		$$
		\begin{aligned}
			d^*(\Delta f df) &= - *d(\Delta f *df) \\
			&= - * (d\Delta f \wedge *df + \Delta f d*df) \\
			&= - \tri{d\Delta f, df} + (\Delta f)^2.
		\end{aligned}
		$$
		Therefore, it is sufficient to require $\tri{d\Delta f, df} = \tri{\nabla \Delta f, \nabla f} \le 0$. By Bochner's formula, this is equivalent to
		\begin{align}\label{eq:Bochner}
			\tri{\nabla \Delta f, \nabla f}=\frac{1}{2}\Delta |\nabla f|^2+|\nabla^2 f|^2 + \Ric(\nabla f, \nabla f) = |\operatorname{Hess} f|^2 + K \le 0,
		\end{align}
		where the Laplacian $\Delta=d^*d$ is positive in our setting. The equation ~\eqref{eq:Bochner} follows from the curvature condition. As for $\alpha[\varphi_i (*\nu)],$ the calculation is similar. This proves the first result.
		
		For the cross term, we have
		$$
			\alpha[\varphi_i\nu,\varphi_i(*\nu)]=
			(d(\varphi_i\nu),d(\varphi_i(*\nu))+(d(\varphi_i(*\nu)),d(\varphi_i (-\nu)))=0.
		$$
	\end{proof}
	
	\begin{rem}
		Let us consider the case where $f$ is chosen as the Busemann function $B_\theta$. $|\operatorname{Hess} B_\theta|$ represents the absolute value of the curvature $h$ of the horocircle. Consequently, \eqref{eq:Bochner} is equivalent to $h^2 + K \le 0$.
		
		Mazzeo in \cite{M1991} introduced a general method for proving Friedlander's inequality by finding a distance function $f$ such that $\Delta f$ remains a constant. This condition corresponds precisely to the equality case in \eqref{eq:Bochner}. Thus our curvature condition is a weaker assumption for proving such eigenvalue inequalities, with stronger results when domains are simply-connected.
	\end{rem}
	The unique continuation results yield the following lemma:
	
	\begin{lemma}\label{lem:maximal}\cite[Proposition~2.5]{BR2012}
		Let $f\in H^1(\Sigma)$ satisfy $\Delta f=\lambda f$ weakly, with $f|_{I}=0$ and $\partial_n f|_{I}=0$ where $I \subset \partial\Sigma$ is a non-empty relatively open set and $n$ is the outer normal vector of $\partial \Sigma$. Then $f\equiv 0$ in $\Sigma$.
		
	\end{lemma}
	
	We now prove the main theorem for a smooth domain.
	
	\begin{thm}\label{thm:pre-main}
		Let $M$ be a Riemannian surface. If a bounded smooth domain $\Sigma \subset M$ satisfies the curvature condition, then the following inequality holds.
		$$
		\mu_{3-\beta_1} \le \lambda_1, \quad \text{where } \beta_1:=\dim H^1_{dR}(\Sigma).
		$$
		If $\beta_1=0,$ we further have the strict inequality:
		$$
		\mu_{3}<\lambda_1.
		$$
	\end{thm}
	
	\begin{proof}
		Let $\eta_1\le \eta_2\le\cdots $ be the eigenvalues of $ A$. The $1$-forms $\{\varphi_k\nu,\varphi_k(*\nu)\}$ generate a $2$-dimensional subspace $\Lambda_k\subset L^2\Omega^1(\Sigma),$ and these two $1$-forms constitute an orthonormal basis of $\Lambda_k.$ Let $u=x\varphi_k\nu+y\varphi_k (*\nu) \in\Lambda_k,$ we have
		\begin{align}
			\alpha[u]&=x^2\alpha[\varphi_k\nu]+y^2\alpha[\varphi_k(*\nu)]+2xy\alpha[\varphi_k\nu,*\varphi_k\nu]\le \lambda_k (x^2+y^2)=\lambda_k \|u\|_{L^2}^2.\label{eq:u}
		\end{align}
		
		Furthermore, for any $v\in H^1\Omega_D^1(\Sigma)$ such that $Av=\lambda_kv,$ we have $\alpha[v,u]=(Av,u)=\lambda_k(v,u),$ and consequently,
		$$
		\alpha[u+v]=\alpha[u]+2\alpha[u,v]+\alpha[v]\le \lambda_k(\|u\|_{L^2}^2+2(u,v)+\|v\|_{L^2}^2)=\lambda_k \|u+v\|_{L^2}^2.
		$$ 
		
		Let $E_k$ denote the subspace $\Lambda_k+\ker (A-\lambda_k) $. We derive
		$$
		\eta(u)\le \lambda_k,\quad \forall u\in E_k.
		$$
		Recall that from Proposition~\ref{prop:specDec} we have
		$$\ker (A-\lambda_k)=d\ker (\Delta_N-\lambda_k)\oplus d^*\ker (\Delta_D-\lambda_k).$$
		Assuming $u\in \Lambda_k\cap d^*\ker(\Delta_D-\lambda_k),$ we have $u=0$ by Lemma~\ref{lem:maximal}. Therefore, we obtain that 
		$$\dim E_k\ge \dim\Lambda_k+ \dim \ker (\Delta_D-\lambda_k)=2+m_D(\lambda_k),$$
		where $m_D(\lambda_k)$ denotes the multiplicity of $\lambda_k$ in $\Spec\Delta_D.$
		
		Let $k=1,$ as the lowest Dirichlet eigenvalue is simple \cite{CouHil53}, we have $\dim E_1\ge 3,$ and then $0\le \eta_1,\eta_2,\eta_3\le \lambda_1.$ As $\dim \ker A=\beta_1$ by Theorem~\ref{thm:isomorphism}, we have $3-\beta_1$ eigenvalues of $A$ in $(0,\lambda_1].$ While $\lambda_1$ is the simple lowest eigenvalue of $\Delta_D,$ $\Delta_N$ has at least $2-\beta_1$ eigenvalues in $(0,\lambda_1]$ by the spectrum decomposition in Theorem~\ref{thm:mainDecomposition}. Taking into account that $\mu_1=0,$ we derive our conclusion that $\mu_{3-\beta_1}\le \lambda_1.$
		
		If $\Sigma$ is simply-connected, we will show that $\Lambda_1\cap \ker (A-\lambda_1)=\emp.$ Then we derive $\eta_1,\eta_2<\lambda_1,$ thus $\mu_3<\lambda_1.$ 
		
		Otherwise, suppose $u\in \Lambda_1\cap \ker (A-\lambda_1)$ is non-zero. We can assume
		$$
		x\varphi_1\nu+y\varphi_1 (*\nu) = u = d\psi+cd^*\Phi_1,\quad c\in \R,
		$$ 
		where $\psi\in H^1(\Sigma)$ is the Neumann eigenfunction and $\Phi_1=\varphi_1\vol_M.$ Then, 
		$$xd\varphi_1\wedge \nu+y d\varphi_1 \wedge (*\nu)+y \varphi_1 d(*\nu) = du =cdd^*\Phi_1=c\lambda_1\Phi_1.$$
		Restricted on $\partial\Sigma,$ we have
		$$
		(-x\tri{d\varphi_1,*\nu}\vol_M + y\tri{d\varphi_1,\nu}\vol_M)|_{\partial_\Sigma} = du|_{\partial\Sigma}=c\lambda_1(\varphi_1\vol) |_{\partial\Sigma}=0,
		$$
		thus at the boundary, we have
		\begin{align}
			\tri{d\varphi_1,y\nu-x(*\nu)}=0. \label{eq:normal}
		\end{align}
		
		Let $\omega:=y\nu-x(*\nu),$ we have
		$$d\omega=-xd(*\nu)=-x\Delta f \vol_M,\quad d(*\omega) = d(y(*\nu)+x\nu)=y\Delta f \vol_M,$$
		and then
		\begin{align}
			\int_{\partial\Sigma} i^*\omega=\int_{\Sigma} d\omega=-x\int_{\Sigma} \Delta f\vol_M,\quad \int_{\partial \Sigma} i^*(*\omega)=\int_{\Sigma} d(*\omega) = y\int_{\Sigma}\Delta f\vol_M. \label{eq:Stokes}
		\end{align}
		
		We now claim that there exists a relatively open set $I\subset \partial\Sigma$ such that $\omega|_{I}$ contains normal components. Then from equation \eqref{eq:normal}, we have $d\varphi_1|_{I}=0,$ since $d\varphi_1$ only has normal components at the boundary as $\varphi_1|_{\partial\Sigma}=0$. By Lemma~\ref{lem:maximal}, it follows that $\varphi_1\equiv 0,$ which leads to a contradiction.
				
		Otherwise, since $\omega$ is continuous, we can suppose that $\omega$ only has tangential components at the boundary. Then 
		\begin{align}
			\left|\int_{\partial\Sigma} i^*\omega\right|=(x^2+y^2)|\partial\Sigma|>0,\quad \int_{\partial \Sigma} i^*(*\omega)=0.\label{eq:positiveTerms}
		\end{align}
			
		By equations \eqref{eq:Stokes} and \eqref{eq:positiveTerms}, we can deduce that \( y = 0 \), implying that \( \omega = -x(*\nu) \) is tangent to \( \partial\Sigma \). This means that \( \nu = df \) is normal to \( \partial\Sigma \). Since \( |\nu| \equiv 1 \), \( \nu \) preserves its sign at the boundary, i.e., it points either inward or outward. Clearly, its index in \( \Sigma \) is zero, as it does not vanish anywhere, and \( \chi(\Sigma) = 1 \), since \( \Sigma \) is simply connected. However, this contradicts the Poincaré-Hopf index Theorem \cite{GP2010} for manifolds with boundary, where the theorem points out that these two quantities should be equal.
	\end{proof}

	\begin{rem}\label{rem:firstTry}
		We initially aimed to prove a more general Friedlander-type inequality, namely 
		\(\mu_{k+2} \le \lambda_k\). However, the cross term 
		\(\alpha[\varphi_i \nu, \varphi_j(*\nu)]\) does not vanish, preventing us 
		from deriving an inequality analogous to \eqref{eq:u}. Fortunately, this 
		cross term vanishes when we restrict our attention to the first Dirichlet 
		eigenvalue \(\varphi_1\).
	\end{rem}
	
	\begin{proof}[Proof of Theorem~\ref{thm:main}]
		We only need to adopt a Lipschitz approximation by smooth domains from the exterior while maintaining the Betti number unchanged. The result follows from Theorem~\ref{thm:pre-main} and the approximation. The inequality of eigenvalues remains valid by semi-continuity of Neumann eigenvalues under this perturbation, see Theorem~\ref{thm:NeumannConverge} in Appendix. However, the strict inequality cannot be preserved.
	\end{proof}
	
	\begin{rem}
		The difficulty in extending our result to higher dimensions stems from the lack of analogous test-form constructions and the challenge of combining the spectra of $\Delta_D$ and $\Delta_N$ as the dimension increases.
		
		However, one may attempt to define and prove the result for Dirichlet and Neumann problems in $n-1$ and $n+1$ dimensions, respectively, when the manifold is $2n$-dimensional.
	
	\end{rem}
	
	\subsection{Applications on standard models} According to Theorem~\ref{thm:main}, the comparison of the Neumann and Dirichlet eigenvalues can be reduced to finding an appropriate distance function that satisfies the inequality $\eqref{eq:curvatureCondition}$ in the curvature condition. We begin with the hyperbolic space.
	\begin{example}\label{eg:spaceForm}
		We provide an explicit expression of the Busemann function $B_\theta(z)$ to prove that it satisfies the curvature condition. Consequently, the eigenvalue inequalities in Theorem~\ref{thm:main} holds for any bounded Lipschitz domain in the hyperbolic space.
		
		To simplify the calculation, we employ the upper-plane model with $K=1$
		$$
		\H^2:=\H^2(-1)=\{z=x+iy\in \C\mid y>0\},\quad ds^2=\frac{dx^2+dy^2}{y^2}.
		$$
		Without loss of generality, we select the infinity point $\infty\in \partial \H^2$ for the Busemann function as
		$$
		B_\infty(z)=-\ln y\in C^\infty(\H^2),\quad |\nabla B_\infty|\equiv 1.
		$$
		By direct calculation, one can derive
		$$
		\operatorname{Hess} B_\infty=\begin{bmatrix}
			\frac{1}{y^2}&0\\0&0
		\end{bmatrix},\quad |\operatorname{Hess} B_\infty|^2=\left(g^{xx}\right)^2\left(\frac{1}{y^2}\right)^2\equiv 1,
		$$
		which ensures that the condition $K+|\operatorname{Hess} B_\infty|^2\le 0$ in any bounded domain $\Sigma\subset \H^2 $.
		
	\end{example}
	
	For the calculation of eigenvalues in hyperbolic space, see \cite{BF2017,B2023,M1991} for related results. In the following example, we provide the explicit eigenvalues of a ring-shaped domain on a flat cylinder to verify our inequality and show the feasibility of general one mentioned in Remark~\ref{rem:firstTry}.
	
	\begin{example}\label{eg:cylinder}
		For $\Sigma=S^1\times[0,\pi]\subset S^1\times \R=M,$ by the separation of variables, one easily derives that
		$$
		\lambda_{i,j}=i^2+j^2,\quad \mu_{i,k}=i^2+k^2,\quad i\in \Z,\,j\in \Z_+,\,k\in \N 
		$$
		constitute all eigenvalues. As $\mu_{i,j}=\lambda_{i,j}$ for all $i\in \Z,$ $j\in \Z_+$ and $\mu_{i,0}\le \lambda_{i,j},$ we can identify $m+1$ Neumann eigenvalues smaller than the $m$-th Dirichlet eigenvalue for all $m\in \Z_+$. Thus, after sorting the eigenvalues, we obtain
		$$
		\mu_{m+1}\le \lambda_{m},\quad \foa m\ge 1.
		$$
	\end{example}
	
	Before proceeding to next example, we first prove Theorem~\ref{thm:logConvex} that provides a method to find suitable twisted product spaces.
	
	\begin{proof}[Proof of Theorem~\ref{thm:logConvex}]
		It is sufficient to select function $f(r,\theta)=r$ as the distance function for the curvature condition. It remains to verify the inequality~\eqref{eq:curvatureCondition}. By the calculation in \cite{P2006}, we require
		$$
		K+|\operatorname{Hess} f|^2=-\frac{\partial_r^2\phi}{\phi}+\left(\frac{\partial_r\phi}{\phi}\right)^2=-\partial_r^2\log \phi\le 0.
		$$
		Therefore, provided that $\phi(r,\theta)$ is log-convex in $r$ for all $\theta\in S^1$, $f(r,\theta)=r$ is a function that satisfies the curvature condition on $M.$	
		
		The remaining part follows from Theorem~\ref{thm:main}.	
	\end{proof}
	
	By Theorem~\ref{thm:logConvex}, we can prove the results for cusps and collars, which are essential fundamental structures in hyperbolic geometry \cite{BMM2018,B1992}.
	
	\begin{example}\label{eg:cusp}
		Let $M$ be a cusp equipped with the standard isothermal metric \cite{W2007}
		$$
		g(z)=\left(\frac{|dz|}{|z|\log|z|}\right)^2,\quad z\in B_{r_0}(0)\setminus\{0\},\quad r_0<1.
		$$
		By setting $z=r e^{i\theta},$ we obtain
		$$
		g(r,\theta)=\left(\frac{1}{r\log r}\right)^2\left(dr^2+r^2d\theta^2\right),\quad r\in (0,r_0).
		$$
		Let $\rho=-\log (-\log r),$ we derive
		$$
		g(\rho,\theta)=d\rho^2+ e^{2\rho} d\theta^2,\quad \rho\in (-\infty,-\log(-\log r_0)).
		$$
		As $\log e^{\rho}=\rho$ is convex, the curvature condition is ensured by $f(r,\theta)=r$ for any bounded Lipschitz domain $\Sigma$ on a cusp. Consequently, the eigenvalue inequalities in Theorem~\ref{thm:main} hold for $\Sigma$.		
	\end{example}
	
	\begin{example}\label{eg:collar}
		Let $M$ be a collar. According to \cite{B1992}, it is isometric to the cylinder $I\times S^1$ with metric
		$$
		ds^2 = d\rho^2 + l_0^2\cosh^2\rho d\theta^2,
		$$
		where $l_0$ represents the length of the simple closed geodesic that corresponds to the collar, and $I\subset \R$ is an interval containing $0$.
		
		By the calculations
		$$
		(\log \cosh \rho)'=\frac{\sinh\rho}{\cosh\rho},\quad (\log \cos\rho)''=\frac{\cosh^2\rho-\sinh^2\rho}{\cosh^2\rho}=\frac{1}{\cosh^2\rho}\ge 0,
		$$
		and hence $l_0\cosh\rho$ is log-convex. Thus, the eigenvalue inequalities in Theorem~\ref{thm:main} hold for every bounded Lipschitz domain in a collar, by Theorem~\ref{thm:logConvex}. 
	\end{example}
		
	The eigenvalue inequality is applicable to bounded Lipschitz domains in funnels because funnels constitute half of the collars.
	
	Finally, we show that the eigenvalue inequalities in Theorem~\ref{thm:main} hold on certain minimal surfaces, including the helicoid and catenoid, which are among the most common and useful minimal surfaces. For instance, these surfaces serve as comparison models in the classification studies of minimal disks \cite{CM2011}.
	
	\begin{example}\label{eg:helicoid}
		Let $M\hookrightarrow \R^3$ be a helicoid parameterized by
		$$
		(x_1,x_2,x_3)=(t\cos s,t\sin s, cs),\quad c>0.
		$$
		We have the following
		$$
		\pfrac{s}=(-t\sin s,t\cos s,c),\quad \pfrac{t}=(\cos s,\sin s,0),
		$$
		$$
		g(t,s)=dt^2+(t^2+c^2)ds^2.
		$$
		$$
		\log\left(\sqrt{t^2+c^2}\right)''=\frac{(t^2+c^2)-2t^2}{(t^2+c^2)^2}\ge 0\quad \Leftrightarrow \quad |t|\le c.
		$$
		By Theorem~\ref{thm:logConvex}, the eigenvalue inequalities in Theorem~\ref{thm:main} hold for every bounded Lipschitz domain $\Sigma\subset M$ contained in $\{|t|<c\}.$ 
	\end{example}

	\begin{example}\label{eg:catenoid}
		Let $M\hookrightarrow \R^3$ be a catenoid parameterized by
		$$
		(x_1,x_2,x_3)=\left(a\cosh \frac{u}{a} \cos v, a\cosh\frac{u}{a}\sin v, u\right),\quad a>0.
		$$
		We have the following
		$$
		\pfrac{u}=\left(\sinh\frac{u}{a}\cos v,\sinh\frac{u}{a}\sin v,1\right),\quad \pfrac{v}=\left(-a\cosh \frac{u}{a}\sin v,a\cosh \frac{u}{a}\cos v,0\right),
		$$
		$$
		g(u,v)=\cosh^2\frac{u}{a}\left(du^2+a^2dv^2\right).
		$$
		Let $r=a\sinh\frac{u}{a},$ $\theta=v,$ we have
		$$
		ds^2=dr^2+(r^2+a^2)d\theta^2.
		$$
		Similarly we derive
		$$
		\log\left(\sqrt{r^2+a^2}\right)''\ge 0\quad \Leftrightarrow \quad |r|\le a.
		$$
		By Theorem~\ref{thm:logConvex}, the eigenvalue inequalities in Theorem~\ref{thm:main} hold for every bounded Lipschitz domain $\Sigma\subset M$ such that $\Sigma\subset \{|r|<a\}.$
	\end{example}
	
	\backmatter
	
	\bmhead{Acknowledgements}
	
	We extend our gratitude to Prof. Zhiqin Lu for discussions on differential forms, Prof. Zuoqin Wang and his students for discussions on regular domain perturbation and strict inequality, and Jin Sun for discussions on cross terms of the Dirichlet form.
	
	\section*{Declarations}
	
	\begin{itemize}
		\item \textbf{Funding:} B. Hua was supported by the National Natural Science Foundation of China (No. 12371056) and the Shanghai Science and Technology Program (Project No. 22JC1400100).
		\item \textbf{Conflict of interest:} The authors declare no conflict of interest.
		\item \textbf{Ethics approval and consent to participate:} Not applicable, as this study does not involve human or animal subjects.
		\item \textbf{Consent for publication:} Not applicable, as this study does not include any individual personal data or images requiring consent.
		\item \textbf{Data availability:} Not applicable, as this work does not involve the use of datasets.
		\item \textbf{Materials availability:} Not applicable, as no materials are used or distributed in this study.
		\item \textbf{Code availability:} Not applicable, as this study does not involve the use of code or software.
		\item \textbf{Author contributions:} The authors contributed to the conception, writing, and revision of the manuscript.
		\item \textbf{Ethical Statement:} This work complies with the ethical standards of scientific research. The manuscript is original, has not been published elsewhere, and has not been submitted simultaneously to any other journal. No part of this work is plagiarized; all borrowed ideas and direct quotations from other works have been appropriately cited with quotation marks where verbatim texts were used. 
	\end{itemize}

\begin{appendices}
	\section{Approximation of Lipschitz domains by smooth domains}
	In the appendix, we prove that a sequence of smooth domains can be chosen to approximate a Lipschitz domain $\Sigma$ from the outside and that the eigenvalues are lower semi-continuous under this convergence. For a comprehensive introduction, the reader may refer to \cite{H2005}.
	
	For the Dirichlet eigenvalue problem, the results were established in \cite{BV1965}. Our focus here is on Neumann eigenvalues. We first select an appropriate sequence of smooth domains.
	
	\begin{lemma}
		For any bounded Lipschitz domain $\Sigma \subset M$, there exists a sequence of smooth domains $\{\Sigma_m\}_{m=1}^\infty$ such that $\Sigma \subset \Sigma_m$ and 
		\[
		m(\Sigma_m \setminus \Sigma) \to 0 \quad \text{as} \quad m \to \infty.
		\]
		Furthermore, $\partial\Sigma_m$ ($\partial\Sigma$ resp.) can be expressed by smooth (Lipschitz resp.) functions locally in the same coordinate chart with uniform Lipschitz constant $C(L)$.
	\end{lemma}
	
	\begin{proof}
		The results for the Euclidean space are presented in \cite{D1976, AC2024}, where the proof uses small open balls centered at the boundary. This method can be easily applied to manifolds. We briefly explain a few minor distinctions between statements on Euclidean space and manifold, while showing the key steps of the proof, omitting specific details.
		
		In the Euclidean space, a Lipschitz domain $\Sigma \subset \mathbb{R}^n$ is defined such that for each point $x \in \partial \Sigma$, there exists a small open ball $B_x$ centered at $x$ for which $B_x \cap \partial \Sigma$ is the graph of a Lipschitz function with constant $L = L(\Sigma)$.
		
		The situation is similar for a Riemannian manifold; however, we need to apply the exponential map on the tangent space. Specifically, for a Lipschitz domain $\Sigma \subset M$, for each $x \in \partial \Sigma$, there exists a small open neighborhood $B_x$ such that $\exp_x^{-1}(B_x \cap \partial \Sigma)$ is a graph of a Lipschitz function in the tangent space $T_x M$ with a constant $L=L(\Sigma)$. This function is defined on a hyperplane $E_x \subset T_x M$ with values lying in the direction of the orthogonal complement $E_x^\perp$ in $T_x M$. Let $n_x \in E_x^\perp$ be a unit vector pointing outward from $\Sigma$.
		
		Let $\pi : T_x M \to E_x$ be a projection map. We define $B_x' = \pi(\exp_x^{-1}(B_x \cap \partial \Sigma))$. Subsequently, there exists a Lipschitz function $\phi: B_x' \to \mathbb{R}$ such that $\phi(0) = 0$, and the boundary of $\Sigma$ near $x$ is described by
		\[
		B_x \cap \partial \Sigma = \{\exp_x(y + \phi(y) n_x) \mid y \in B_x'\},
		\]
		while the interior of $\Sigma$ near $x$ is given by
		\[
		B_x \cap \Sigma = \{\exp_x(y + u n_x) \mid y \in B_x',\, u < \phi(y),\, y + u n_x \in \exp_x^{-1}(B_x \cap \Sigma)\}.
		\]
		In fact, we can select $B_x$ appropriately such that $\exp_x^{-1}(B_x)$ forms a cylinder, with base $B_x'$ and the axis along $E_x^\perp$.
		
		By further selecting a smaller neighborhood, we can ensure that for any $\eps > 0$, we have
		\[
		\|(d\exp_x)_y - \mathrm{id}_{xy}\| \leq \eps \quad \text{for all } y \in \exp_x^{-1}(B_x),
		\]
		where $\mathrm{id}_{xy} : T_y T_x M \to T_{\exp_y} M$ is an isometry that identifies two Euclidean spaces, and the norm is considered in the sense of linear maps. We require $\id_{x\bullet}:TT_xM\rightarrow TM$ to be smooth.
		
		Now, we consider that $y \in B_{x_1} \cap B_{x_2}$. At the point $\exp_{x_1}^{-1}(y)$, the map $d\exp_{x_2}^{-1} \circ d\exp_{x_1}$ is almost orthogonal, which means
		\[
		\|d\exp_{x_2}^{-1} \circ d\exp_{x_1} - \mathrm{id}_{x_2 \exp_{x_2}^{-1}(y)}^{-1} \circ \mathrm{id}_{x_1 \exp_{x_1}^{-1}(y)}\| \leq 3\eps.
		\]
		Here, $\mathrm{id}_{x_2 \exp_{x_2}^{-1}(y)}^{-1} \circ \mathrm{id}_{x_1 \exp_{x_1}^{-1}(y)} : T_{\exp_{x_1}^{-1}(y)} T_{x_1} M \to T_{\exp_{x_2}^{-1}(y)} T_{x_2} M$ is an orthogonal map, as it is isometric between two Euclidean spaces. We select the identification map $\mathrm{id}_{x_1 x_2} : T_{x_1} M \to T_{x_2} M$ such that $\mathrm{id}_{x_1 x_2}(E_{x_1}) = E_{x_2}$ and $\mathrm{id}_{x_1 x_2}(n_{x_1}) = n_{x_2}$.
		
		\begin{figure}[htbp]
			\centering
			\includegraphics[width=0.8\textwidth]{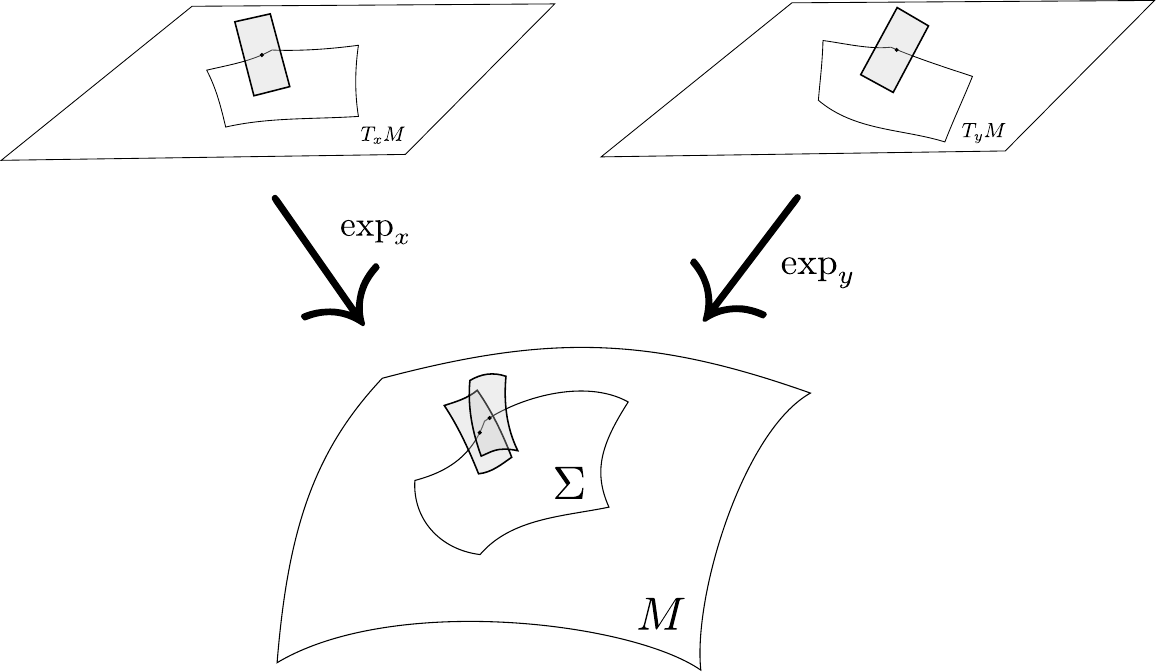}
			\caption{Graphs in two charts differ by an almost orthogonal transformation.}
		\end{figure}
		
		In the Euclidean space, all identification maps are natural, and the exponential map is simply the identity. Consequently, this definition is consistent when the manifold is the Euclidean space.
		
		After clarifying these definitions, we now outline the proof. Initially, we select a finite open cover $\bigcup_{i=0}^N B_i\supset  \Sigma,$ where $\{B_{i}\}_{i=1}^N$ are small neighborhoods at the boundary, satisfying the conditions discussed above. $B_0\subset \Sigma$ is a large open set that covers $\Sigma\setminus \bigcup_{i=1}^NB_ i $. We then consider the unity partition $\{\xi_i\}_{i=0}^N$ belonging to this open cover.
		
		For each neighborhood at the boundary, we have the Lipschitz map $\phi_i$ defined on $B_i'\subset E_i\subset T_{x_i}M.$ We can construct smooth functions $\{\phi_i^{(m)}\}$ by mollifiers such that 
		$$\frac{L}{m}\le \phi_i^{(m)}-\phi_i \le \frac{3L}{m},\quad |d\phi_i^{(m)}|\le L.$$
		Combining these maps with the unity partition, we can obtain the boundary defining function $F,\{F_m\}_{m=1}^\infty$ from $\bigcup_{i=0}^N B_i$ to $\R,$ such that
		$$
		\{F=0\}=\partial\Sigma,\quad \{F<0\}=\Sigma,
		$$
		and the smooth domains
		$$
		\Sigma_m:=\{F_m<0\}
		$$
		satisfy the condition precisely with the boundary $\partial\Sigma_m=\{F_m=0\}.$
		
		Furthermore, $\partial\Sigma_m$ and $\partial\Sigma$ can be expressed using the same local coordinate system in $B_{i}.$ This implies that $\partial \Sigma_m$ can also be expressed by smooth function $f_i^{(m)}$ defined on $B_{i}'.$ It is necessary to prove that $F_m$ in $B_i$ is not degenerate in the direction of $n_i$ and then apply the implicit function theorem. This part in the Euclidean space is proved by the transversality property in \cite{AC2024}, where they primarily utilize the fact that the transition map between Lipschitz graphs in different charts is rigid, that is, a translation plus an orthogonal map. As previously established, we have shown that the transition map between $B_i$ and $B_j$ is almost orthogonal. Consequently, the transversality property can be proven in manifolds in a similar way.
		
		Given that $|d\phi_i^{(m)}|\le L,$ $|d\xi_i|\le C$, we can easily show that $|df_i^{(m)}|\le C(L)$ uniformly in $m$ using the implicit function theorem.
	\end{proof}
	
	The uniform Lipschitz constant is utilized in the following lemma.
	
	\begin{lemma}
		The collection $\{\Sigma_m\}_{m=1}^\infty$ obtained by previous method satisfies the $\eps$-cone condition, where $\eps>0$ is sufficiently small depending on $L.$ Consequently, there exists a linear continuous extension operator $P_m:H^1(\Sigma_m)\rightarrow H^1(\Sigma_1)$ and $\|P_m\|\le K$ for some $K>0$ independent of $m.$
		
	\end{lemma}
	
	\begin{proof}
		The $\eps$-cone condition is described in many textbooks, and we refer to \cite[Definition 2.4.1]{HP2018} for an example. Define the cone $C(y,\xi,\eps)\subset \R^N$ for $y\in \R^N,$ $\xi\in S^{N-1}$ and $\eps>0:$
		$$
		C(y,\xi,\eps):=\{z\in \R^N,\, \tri{z-y,\xi}\ge \cos(\eps)|z-y|,\,|z-y|\in (0,\eps)\}.
		$$
		The domain $\Sigma$ is said to satisfy the $\eps$-cone condition if for all $x\in \partial\Sigma,$ there exists $\xi_x\in S^{N-1}$ such that
		$$
		C(y,\xi_x,\eps)\subset \Sigma,\quad \foa y\in \Sigma \cap B_\eps(x),
		$$
		where $B_\eps(x)$ denotes an open ball with radius $\eps$ centered at $ x $.
		
		For manifolds, it is necessary to adopt the exponential map to express the condition precisely. The cone can only be defined on the tangent space; therefore, the cone condition is transformed into
		$$
		C(y,\xi_x,\eps)\subset \exp_x^{-1} \Sigma,\quad \foa y\in (\exp_x^{-1}  \Sigma) \cap B_\eps(0).
		$$
		
		\begin{figure}[htbp]
			\centering
			\includegraphics[width=0.4\textwidth]{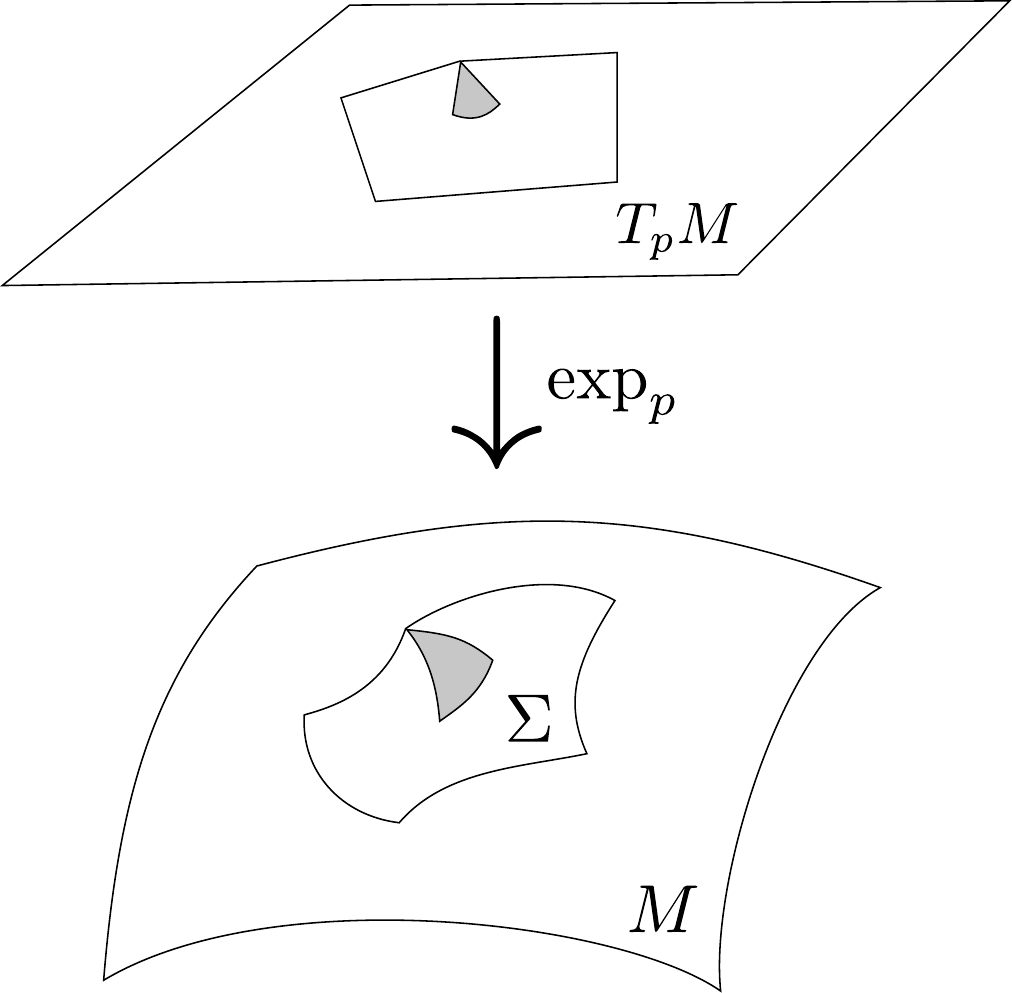}
			\caption{Cone condition on manifolds.}
		\end{figure}
		
		According to \cite[Proposition 3.7.2]{HP2018}, with the proof provided in \cite{C1975}, given a larger bounded domain $B,$ there exists $k>0$ such that for all domains $\Sigma\subset B$ with the $\eps$-cone property, there exists a linear continuous extension operator $P_\Sigma:H^1(\Sigma)\rightarrow H^1(B)$ such that $\|P_\Sigma\|\le K$. Although the statement is presented in Euclidean space, because the proof is proceeded locally and uses the partition of unity, it can be easily extended to manifolds with the previous definitions.
		
		We assert that $\{\Sigma_m\}_{m=1}^\infty$ satisfies the $\eps$-cone condition. This follows directly from their global Lipschitz constant $C(L)$. Consequently, the continuous extension operator is naturally derived.
	\end{proof}
	
	Finally, we prove the lower semi-continuity of the Neumann eigenvalue under this perturbation. 
	
	\begin{thm}\label{thm:NeumannConverge}
		Let $\mu_k^{(m)}$ be the $k$-th Neumann eigenvalue of $\Sigma_m,$ and $\mu_k$ be the $k$-th Neumann eigenvalue of $\Sigma $. We have
		$$
		\mu_k\le \liminf_{m\rightarrow \infty} \mu_k^{(m)}.
		$$
	\end{thm}
	
	\begin{proof}
		Take $L^2$-orthonormal eigen-basis $\{\psi_i\}_{i=1}^\infty$ in $H^1(\Sigma),$ such that  
		$$
		\left\{\begin{aligned}
			&\Delta \psi_i=\mu_i \psi_i,\\
			&\pfrac[\psi_i]{n} |_{\partial \Sigma}=0.
		\end{aligned}\right.
		$$
		
		Let $\{\psi_i^{(m)}\}_{i=1}^\infty$ be the $L^2$-orthonormal eigen-basis in $H^1(\Sigma_m).$ Define
		$$v_i^{(m)}:=\psi_i^{(m)}|_{\Sigma}\in H^1(\Sigma).$$ 
		
		For fixed $k\in \Z_+,$ 
		$$
		\|v_k^{(m)}\|_{H^1(\Sigma)}\le \|\psi_k^{(m)}\|_{H^1(\Sigma_m)} = 1 + \mu_k,
		$$
		thus there exists $u_k\in H^1(\Sigma)$ such that 
		$$
		v_k^{(m)}\stackrel{H^1}{\rightharpoonup} u_k,\quad v_k^{(m)}\xrightarrow{L^2} u_k,\quad m\rightarrow \infty.
		$$
		
		We assert that $\{u_k\}_{k=1}^\infty$ constitutes an $L^2$-orthonormal basis. The following equations hold:
		$$
		(v_i^{(m)},v_j^{(m)})_{\Sigma}\rightarrow (u_i,u_j)_{\Sigma},\quad 
		(\psi_i^{(m)},\psi_j^{(m)})_{\Sigma_m}=\delta_{ij},
		$$
		$$
		\begin{aligned}
			(v_i^{(m)},v_j^{(m)})_{\Sigma}&=(\psi_i^{(m)},\psi_j^{(m)})_{\Sigma_m}-(\psi_i^{(m)},\psi_j^{(m)})_{\Sigma_m\setminus\Sigma}.%\\
			%			&=\delta_{ij}- \int_{\Sigma_m\setminus\Sigma} \psi_i^{(m)}\psi_j^{(m)} d\vol.
		\end{aligned}
		$$
		
		As
		$$
		\left|(\psi_i^{(m)},\psi_j^{(m)})_{\Sigma_m\setminus\Sigma}\right|\le \|\psi_i^{(m)}\|_{L^2(\Sigma_m\setminus\Sigma)}^2|\le \|\psi_j^{(m)}\|_{L^2(\Sigma_m\setminus\Sigma)}^2,
		$$
		it suffices to show that $\|\psi_k^{(m)}\|_{L^2(\Sigma_m\setminus\Sigma)}\rightarrow 0$ for all $k\in \Z_+.$
		
		First of all, we shall show that $\|\psi_k^{(m)}\|_{L^p(\Sigma_m)}$ is bounded in $m$ for $p>2$. Let $P_m$ denote the extension map. By applying the Sobolev inequality, we obtain
		$$
		\begin{aligned}
			\|\psi_k^{(m)}\|_{L^p(\Sigma_m)}&\le \|P_m[\psi_k^{(m)}]\|_{L^p(\Sigma_1)}\\
			&\le C_p   \|P_m[\psi_k^{(m)}]\|_{H^1(\Sigma_1)}\\
			&\le C_pK \|\psi_k^{(m)}\|_{H^1(\Sigma_m)} \\
			&= C_pK(1+\mu_k).
		\end{aligned}
		$$
		
		As a result, we derive
		$$
		\begin{aligned}
			\|\psi_k^{(m)}\|_{L^2(\Sigma_m\setminus\Sigma)} &= \|1_{\Sigma_m\setminus \Sigma}\psi_k^{(m)}\|_{L^2(\Sigma_m)} \\
			&\le m(\Sigma_m\setminus \Sigma)^\frac{1}{q} \|\psi_k^{(m)}\|_{L^p(\Sigma_m)}\rightarrow 0,&\frac{1}{q}+\frac{1}{p}=\frac{1}{2}.
		\end{aligned}
		$$
		
		Consequently, $\{u_k\}_{k=1}^\infty$ constitutes an $L^2$-orthonormal basis. Therefore, by the min-max principle,
		$$
		\begin{aligned}
			\mu_k &= \inf_{\begin{subarray}{c}
					E\subset H^1(\Sigma)\\
					\dim E=k
			\end{subarray}}\sup_{u\in E\setminus\{0\}} \frac{\|\nabla u\|_{L^2(\Sigma)}}{\|u\|_{L^2(\Sigma)}}\\
			&\le \sup_{\sum_{i=1}^k a_i^2=1}\left\|\sum_{i=1}^k a_i \nabla u_i\right\|_{L^2(\Sigma)}\\
			&\le \liminf_{m\rightarrow\infty} \sup_{\sum_{i=1}^k a_i^2=1} \left\|\sum_{i=1}^k a_i \nabla v_i^{(m)}\right\|_{L^2(\Sigma)} & v_i^{(m)}\rightharpoonup u_i\\
			&\le \liminf_{m\rightarrow\infty} \sup_{\sum_{i=1}^k a_i^2=1} \left\|\sum_{i=1}^k a_i \nabla \psi_i^{(m)}\right\|_{L^2(\Sigma_m)}\\
			&=\liminf_{m\rightarrow \infty} \mu_k^{(m)}.
		\end{aligned}		
		$$
	\end{proof}
	
	Now we return to the main theorem of this paper. As $\Sigma$ satisfies the curvature condition, we can select $\Sigma_1$ sufficiently small such that all the elements of $\{\Sigma_m \}_{m=1}^\infty$ satisfy the condition. Through the approximation method, it is evident that $\Sigma_m$ is homeomorphic to $\Sigma,$ and thus, the Betti numbers remain unchanged during this perturbation. 
	
	Given that $\mu_{3-\beta_1}^{(m)}\le \lambda_{1}^{(m)},$ we obtain
	$$
	\mu_{3-\beta_1} \le \liminf_{m\rightarrow \infty} \mu_{3-\beta_1}^{(m)}\le \liminf_{m\rightarrow \infty}\lambda_{1}^{(m)}=\lambda_1,
	$$
	which concludes the proof of the Lipschitz domain.
	
	\begin{rem}
		In fact, it is proved in \cite[Proposition 1.1]{K2014} that for Neumann eigenvalues, we have
		$$
		\limsup \mu_k^{(m)}\le \mu_k,\quad \foa k\ge 0.
		$$ 
		Combined with Theorem~\ref{thm:NeumannConverge}, we can further prove the convergence of Neumann eigenvalues. This was also documented in \cite[Proposition~A.9]{CGGS2023}, where it showed that the bounded extension operator is essential for the convergence.
		
	\end{rem}
	It is noteworthy that in this section, our focus is on perturbations that ``globally'' converge to the domain $\Sigma,$ rather than a thin tube connecting two parts of domains $\Sigma_1$ and $\Sigma_2$ or perturbations near a point singularity. The results for these kinds of perturbations can be found in \cite{AHH1991,H2005,A1995}.
	
\end{appendices}

\bibliography{IND}

\end{document}